\documentclass[10pt,reqno]{amsart}
\usepackage{mathtools}
\usepackage{amsmath}
\usepackage{amssymb}
\usepackage{amsthm}
\usepackage{amscd}
\usepackage{enumerate}
\usepackage{xcolor}
\usepackage{subcaption}

\usepackage{graphicx}
\usepackage{stmaryrd}
\usepackage[all]{xy}

\usepackage{epsfig}

\usepackage[all]{xy}
\xyoption{poly}
\usepackage{fancyhdr}
\usepackage{wrapfig}
\usepackage{epsfig}

\usepackage{amssymb,amsmath}
\usepackage{wrapfig}
\usepackage{amscd}
\usepackage{amsrefs}
\usepackage{graphicx}
\usepackage{hyperref}

\def\text#1{\hbox{#1}}
\def\Z{{\mathbb Z}}
\def\R{{\mathbb R}}
\def\N{{\mathbb N}}

\def\C{{\mathbb C}}
\def\Q{{\mathbb Q}}

\def\O{{\mathcal O}}

\def\N{{\rm N}}

\def\logp{{\rm log}^\dagger}

    \DeclareFontFamily{U}{wncy}{}
    \DeclareFontShape{U}{wncy}{m}{n}{<->wncyr10}{}
    \DeclareSymbolFont{mcy}{U}{wncy}{m}{n}
    \DeclareMathSymbol{\Sh}{\mathord}{mcy}{"58}

\newtheorem{theorem}{Theorem}[section]
\newtheorem{definition}[theorem]{Definition}
\newtheorem{lemma}[theorem]{Lemma}
\newtheorem{corollary}[theorem]{Corollary}
\newtheorem{proposition}[theorem]{Proposition}

\newtheorem{question}[theorem]{Question}
\newtheorem{remark}[theorem]{Remark}
\newtheorem{example}[theorem]{Example}

\date{\today}
\theoremstyle{definition}
\numberwithin{equation}{section}
\usepackage{pgfplots}
\DeclareMathOperator*{\esssup}{ess\,sup}

\newcommand\blfootnote[1]{%
  \begingroup
  \renewcommand\thefootnote{}\footnote{#1}%
  \addtocounter{footnote}{-1}%
  \endgroup
}
\def \Leb {{\rm Leb}}
\def\N{{\mathbb N}}

\title{Simultaneous $\mathfrak{p}$-orderings and equidistribution} 
\date{}

 \author{Anna Szumowicz}
  
  \address{Caltech, The Division of Physics, Mathematics and Astronomy,
1200 E California Blvd, Pasadena CA 91125}
\email{anna.szumowicz@caltech.edu}
\subjclass[2010]{11N25,11K38,13F20,11D57} 
\keywords{$\mathfrak{p}$-orderings, integer valued polynomials, potential theory}
\begin{document}
\maketitle
\blfootnote{\today}
\begin{abstract}
Let $D$ be a Dedekind domain. Roughly speaking, a simultaneous $\mathfrak{p}$-ordering is a sequence of elements from $D$ which is equidistributed modulo every power of every prime ideal in $D$ as well as possible. Bhargava in \cite{Bh1} asked which subsets of the Dedekind domains admit simultaneous $\mathfrak{p}$-orderings. We give an overview on the progress in this problem. We also explain how it relates to the theory of integer valued polynomials and list some open problems.  
\end{abstract}
\section{Introduction}
\subsection{Integer valued polynomials and test sets}
Let $D$ be a domain and let $F$ be its field of fractions. We say that a polynomial $P(X)\in F[X]$ is \textbf{integer valued} if $P(D)\subseteq D$. 
 The sum, the product and the difference of integer valued polynomials is again integer valued so the set of integer valued polynomials forms a ring
 
 \begin{equation*} 
 \textrm{Int}(D)=\{ f(X)\in F[X]|\ \ f(D)\subseteq D\}. 
 \end{equation*}
 More generally, for any subset $E\subseteq D$ we can consider
 \begin{equation*} 
 \textrm{Int}(E,D)=\{ f\in F[X]| \ \ f(E)\subseteq D\}.  
 \end{equation*}  
 Integer valued polynomials do not necessarily have coefficients in $D$. For example, consider the case $D=\Z$, $F=\Q$. Then, any polynomial of the form 
 \begin{equation*} 
 {X \choose n}=\frac{X(X-1)\ldots (X-n+1)}{n!}
 \end{equation*}
 with $n\in \N$ is integer-valued. In fact, any integer valued polynomial $P(X)\in \mathbb Q[X]$ of degree $n$ can be uniquely written as a linear combination 
 
 \begin{equation*} 
 P(X)=\sum _{i=0}^{n} \alpha _i {X \choose i}, 
 \end{equation*}
 where $\alpha _i\in \Z$. 
 
 This example shows that to check whether a degree $n$ polynomial $P(X)\in \Z[X]$ is integer valued it is enough to check its values on the set $\{ 0,1 \ldots ,n\}$. Such testing sets can be defined in greater generality. Volkov and Petrov \cite{PV} introduced the notion of an $n$-universal set. 
 
 \begin{definition} 
 Let $D$ be a domain and let $F$ be its field of fractions. We say that a finite subset $S\subseteq D$ is \textbf{$n$-universal} if the following holds:
\\ For every $P(X)\in F[X]$ of degree at most $n$, if $P(S)\subseteq D$ then $P(D)\subseteq D$, i.e. $P(X)$ is integer valued polynomial. 
\end{definition} 
We have the related notion of a Newton sequence. 
\begin{definition} 
Let $D$ be a domain. A sequence $a_0,a_1,\ldots , a_n$ is called a Newton sequence if for every $0\leq m \leq n$ the set $\{ a_0, \ldots a_m\}$ is $m$-universal. The integer $n$ is called the length of the Newton sequence.   
\end{definition} 
Using the Lagrange interpolation it is easy to give a lower bound on the cardinality of an $n$-universal set.
\begin{lemma} 
Let $D$ be a domain which is not a field. Then, every $n$-universal subset of $D$ has at least $n+1$ elements. 
\end{lemma}

\begin{proof}Indeed, for any fixed pairwise different elements $d_0,d_1,\ldots,d_n\in D$ we can construct a polynomial $P$ that takes value $0$ on $d_0,\ldots,d_{n-1}$ but a non-integer value $\alpha$ at $d_n$:
\[ P(X)=\alpha\prod_{i=0}^{n-1}\frac{X-d_i}{d_n-d_i}.
\]
Therefore $P$ is not integer valued, so $d_0,\ldots,d_{n-1}$ cannot be an $n$-universal set.\end{proof}

On the other hand any Dedekind domain will contain an $n$-universal set with $n+2$ elements (see Theorem \ref{theorem-n+2}). Therefore, the case of $n$-universal sets of cardinality $n+1$ is particularly interesting. Some Dedekind domains will contain such sets and as we shall see later, many do not. 

\begin{definition}
Let $D$ be a Dedekind domain. A subset $S\subseteq D$ with $|S|=n+1$ is called \textbf{$n$-optimal} if it is $n$-universal.
\end{definition} 

\begin{example} 
The set $\{ x,x+1,\ldots ,x+n\}$ is $n$-optimal in $\Z$ for every $x\in \Z$.  
\end{example} 

The property of being $n$-optimal can be also understood as optimal equidistribution modulo all prime powers, hence the name. We expand more on that in the following section, after reviewing what is known on the minimal cardinality of an $n$-universal set.
Petrov and Volkov \cite{PV} showed that there are no $n$-optimal sets in $\Z[i]$, for large enough $n$. Building on their method, together with Byszewski and Fraczyk, we generalized their result to the ring of integers in any quadratic imaginary number field.
\begin{theorem}\cite{BFS2017} 
\label{theorem-mainqi}
Let $K$ be a quadratic imaginary number field and let $\O_K$ be its ring of integers. Then, there is no $n$-optimal sets in $\O_K$ for large enough $n$.
\end{theorem} 
For general quadratic number fields Chabert and Cahen \cite{CC} proved that there are no $2$-optimal sets except possibly in $\Q(\sqrt{d})$ with $d=-3,-1,2,3,5$ and $d\equiv 1 \mod 8$. 
The proof of Theorem \ref{theorem-mainqi} as well as the original method of Petrov and Volkov heavily relies on the fact that the norm of the field extension $K/\Q$ is convex. This is not the case in any number field beyond the imaginary quadratic extensions of $\Q$ and $\Q$ itself. 
Together with Fraczyk we used a new potential theoretic approach to extended the result to all number fields $K\neq \Q$.  
\begin{theorem}\cite{FS2018}
\label{theorem-main}
Let $K\neq \Q$ be a number field and let $\O_K$ be its ring of integers. There exists $n_0\in \N$ dependent on $K$ such that for any $n\geq n_0$ there is no $n$-optimal sets in $\O_K$. 
\end{theorem}

The picture becomes complete when we combine the above theorem with a very general upper bound on the minimal cardinality of an $n$-universal set.
\begin{theorem}\cite{BFS2017}
\label{theorem-n+2}
Let $D$ be a Dedekind domain. Then for any $n\in \N$ there exists an $n$-universal set in $D$ of size $n+2$. 
\end{theorem}
The proof is an iterative construction using the Chinese Remainder Theorem. From Theorem \ref{theorem-main} we can now deduce
\begin{corollary}\cite{FS2018} 
Let $K\neq \Q$ be a number field. For $n$ large enough, the minimal cardinality of an $n$-universal set in $\O_K$ is $n+2$. 
\end{corollary} 

\subsection{Equidistribution and simultaneous $\mathfrak{p}$-orderings}
The methods of proving the non-existence of large $n$-optimal sets are based on the almost equidistribution property of $n$-optimal sets.
\begin{definition} 
Let $A$ be a ring and let $I$ be an ideal in $A$. A finite subset $E\subseteq A$ is called \textbf{almost uniformly equidistributed} modulo $I$ if for any $a,b\in A$ we have
\begin{equation*} 
|\{x\in E| \ \ x-a\in I\}|-|\{x\in E|\ \ x-b\in I\}|\in \{-1,0,1\}. 
\end{equation*}
\end{definition}
\begin{lemma}\cite{PV}\cite{BFS2017}
\label{lemma-equidistribution}
Let $K$ be a number field and let $\O_K$ be its ring of integers. Let $S\subseteq \O_K$ be a finite subset with $|S|=n+1$. Then $S$ is $n$-optimal if and only if $S$ is almost uniformly equidistributed modulo $\mathfrak{p^l}$ for every prime ideal $\mathfrak{p}$ in $\O_K$ and all $l\in \N$.  
\end{lemma} 

This characterization of  $n$-optimal sets is reminiscent of the notion of a simultaneous $\mathfrak{p}$-ordering, introduced by Bhargava \cite{Bh1}.
\begin{definition} 
Let $E\subseteq \O_K$ and let $\mathfrak{p}$ be a non-zero proper prime ideal in $\O_K$. A sequence $(a_i)_{i\in\N}\subseteq E$ is called a \textbf{$\mathfrak{p}$-ordering} in $E$ if for every $n\in \N$ we have 
\begin{equation*} 
v_E(\mathfrak{p},n):=v_{\mathfrak{p}}\left(\prod_{i=0}^{n-1}(a_i-a_n)\right)=\min_{x\in E}v_{\mathfrak{p}}\left( \prod_{i=0}^{n-1}a_i-x\right),  
\end{equation*} 
where $v_{\mathfrak{p}}$ denotes the additive $\mathfrak{p}$-adic valuation on $K$. 
\end{definition}
The value $v_E(\mathfrak{p},n)$ does not depend on the choice of a $\mathfrak{p}$-ordering. Bhargava defined the generalized factorial as the ideal $n!_E:=\prod_\mathfrak{p} \mathfrak{p}^{v_E(\mathfrak{p},n)}$ where $\mathfrak{p}$ runs over all prime ideals in $\O_K$. 
A sequence of elements in $E$ is called a \textbf{simultaneous $\mathfrak{p}$-ordering} if it is a $\mathfrak{p}$-ordering for every prime ideal $\mathfrak{p}$ in $\O_K$ at the same time. 
One can show that $(a_i)_{i\in \N}\subseteq \O_K$ is a simultaneous $\mathfrak{p}$-ordering in $\O_K$ if and only if the set $\{a_0, \ldots ,a_n\}$ is $n$-optimal (see Lemma \ref{lemma-equidistribution}). In \cite{Bh1,Bh2} Bhargava asked which subsets of Dedekind domains admit simultaneous $\mathfrak{p}$-orderings. In particular, he asked for which number fields $K$, the ring of integers $\O_K$ admits a simultaneous $\mathfrak{p}$-ordering. A partial progress was made by Wood in \cite{Wood} where she showed that there are no simultaneous $\mathfrak{p}$-orderings in $\O_K$ when $K$ is a quadratic imaginary number field. Adam and Cahen \cite{AC} extended this result to any quadratic number field $\mathbb Q(\sqrt{d}), d\in\mathbb Z$ square-free, except for possibly finitely many exceptional $d$'s.
Using a simultaneous $\mathfrak{p}$-ordering, one could construct $n$-optimal sets for every $n\in\N$. Therefore, Theorem \ref{theorem-main} yields:
\begin{corollary}\cite{FS2018} 
\label{corollary-nonexistencepord}
$\Q$ is the only number field whose ring of integers admits a simultaneous $\mathfrak{p}$-ordering. 
\end{corollary} 
This result answers the question of Bhargava \cite{Bh1, Bh2}. We remark that the methods used to prove Corollary \ref{corollary-nonexistencepord} differ substantially from the methods used by Adam, Cahen and Wood. 

\subsection{Notation}
By $|S|$ we denote the cardinality of the set $S$.
For any $x\in\R$ denote by $\lfloor x\rfloor$ the largest integer less than or equal to $x$. Write $N_{K/\Q}$ for the norm of the extension $K/\Q$. We use the standard big-O and little-o notation. We write $B_{\R}(x,r)$ (resp. $B_{\C}(x,r)$) for a ball in $\R$ (resp. $\C$) of radius $r$ around a point $x\in \R$ (resp. $x\in \C$). 
Denote by $\Leb$ the Lebesgue measure on $\mathbb R,\mathbb C$ and their products. 
We denote by $\Delta _K$ the discriminant of a field $K$. 

\subsection{Structure of the paper} 
In Section \ref{section-quadratic imaginary} we give a sketch of the proof of non-existence of large $n$-optimal sets in the ring of integers of quadratic imaginary number fields (Theorem \ref{theorem-mainqi}). 
In Section \ref{section-energy}, we estimate the energy of $n$-optimal sets. In Section \ref{section-numberfields} we describe the methods used to prove Theorem \ref{theorem-main} in \cite{FS2018}. 
In Section \ref{section-openproblems} we state some questions and open problems. 

\section{$n$-optimal sets for quadratic imaginary number fields}
\label{section-quadratic imaginary}
In this section we give a sketch of the proof of non-existence of large $n$-optimal sets in the ring of integers of a quadratic imaginary number field.
\begin{theorem}\cite{BFS2017} 
Let $K$ be a quadratic imaginary number field and let $\O_K$ be its ring of integers. There is no $n$-optimal sets for large enough $n$. 
\end{theorem}
The condition for a subset to be $n$-optimal can be expressed in terms of the energy ideal of a set.
\begin{definition} 
Let $S=\{x_0, \ldots ,x_n\}$ be a finite subset of $\O_K$. The principal ideal
$E(S):=\prod_{i\neq j}(x_i-x_j)$
is called the energy of the set $S$.
\end{definition} 
The energy ideal was called the volume in \cite{PV, BFS2017} but in the subsequent work \cite{FS2018} it became clear that it is the arithmetic analogue of the energy functional in potential theory, hence the new name.
The energy ideal of an $n$-optimal set is minimal possible in the sense that it should divide the energy of any other set of equal size. This can be made more precise using the factorial ideals:
\begin{definition}\cite{Bh1}\cite{L} 
Let $K$ be a number field and let $\O_K$ be its ring of integers. The $K$-factorial of $n$ is defined as the principal ideal
\begin{equation*} 
n!_K=n!_{\O_K}=\prod _{\mathfrak{p}\in \textrm{Spec}(\O_K)}\mathfrak{p}^{w_{\mathfrak{p}}(n)},
\end{equation*} 
where $w_{\mathfrak{p}}(n)=\sum_{i=1}^{\infty}\lfloor\frac{n}{N(\mathfrak{p}^i)}\rfloor$. 
\end{definition} 

\begin{proposition}\cite{BFS2017,PV} 
\label{proposition-energy}
Let $S\subseteq \O_K$ with $|S|=n+1$. Then, the following conditions are equivalent
\begin{enumerate} 
\item $S$ is $n$-optimal,
\item $E(S)=(\prod _{i=1}^ni!_K)^2$,
\item $E(S)$ divides $E(T)$ for any subset $T$ of $\O_K$ with $n+1$ elements. 
\end{enumerate} 

\end{proposition}
In other words, $n$-optimal sets are the sets with $n+1$ elements which "minimize" the energy among all subsets of $\O_K$ with $n+1$ elements. 
\begin{proof}[Proof of Theorem \ref{theorem-mainqi}]
For the sake of contradiction, assume that for any $n_0$ there exists an $n$-optimal set with $n\geq n_0$. Identify $\O_K$ with its image via a fixed embedding $K\to\C$. The idea of the proof is as follows. Let $S$ be an $n$-optimal set. We sketch the proof why $S$ has to be contained in a polygon with $n+o(n)$ points from $\O_K$. With such a fine description of $S$ one can use the prime number theorem for number fields to show that there exists a prime power $\mathfrak{p}^l$ such that $S$ fails to be almost uniformly equidistributed modulo $\mathfrak{p}^l$. Together with Proposition \ref{proposition-energy}, this leads to a contradiction.

To show that an $n$-optimal set is contained in a suitable polygon with $n+o(n)$ points from $\O_K$ we use Proposition \ref{proposition-energy} and a procedure called discrete collapsing. 
Roughly speaking, the collapsing procedure takes a finite subset of $\O_K$ and makes it symmetric about a line $\ell$ by moving the points towards the line as close as possible. For the proof we will only need the formal definition of what it means to be collapsed with respect to a line. 

\begin{definition}\cite{BFS2017} 
Let $K$ be a quadratic imaginary number field. Let $T$ be a finite subset of $\O_K$. Let $l$ be a line in $\C$. The line $\ell$ divides the complex plane into two closed half-planes, say $H_1$ and $H_2$. Distinguish one of them, say $H_1$. The set $T$ is \textbf{collapsed} along the pair $(l,H_1)$ if the following conditions hold:
\begin{enumerate} 
\item Let $m$ be a line which is perpendicular to $\ell$ and contains at least one point from the set $T$. Let $x\in T\cap m$. Then every point in $\O_K$ which lies between $m\cap l$ and $x$ is in $T$. 
\item Let $m$ be a line perpendicular to $\ell$. Then $|H_1\cap m\cap T|-|H_2\cap m \cap T|\in \{0,1\}$.
\end{enumerate} 
\end{definition} 
The discrete collapsing procedure produces collapsed sets and one can show that if the set was not collapsed to begin with, it strictly decreases the norm of the energy \cite{BFS2017}.
One can show that an $n$-optimal set has to be collapsed in every direction. More precisely, from Proposition \ref{proposition-energy} we deduce the following. 

\begin{lemma} 
\label{lemma-collapsing}
Let $K$ be a quadratic imaginary number field and let $\O_K$ be its ring of integers. Let $T$ be an $n$-optimal set in $\O_K$ and let $l$ be a line in the complex plane. Then there exists a line $m$ parallel to $l$ such that the set $T$ is collapsed along the line $m$ for some choice of the distinguished half-plane. 
\end{lemma} 
For the proof we refer to \cite{BFS2017}. 
\begin{definition} 
Let $l$ be a line in the complex plane. A \textbf{strip} along $l$ is a closed domain which is bounded by two lines parallel to $l$ and symmetric with respect to $l$. A strip parallel to $l$ is a strip along a line parallel to $l$. 
\end{definition} 
Let $K=\Q(\sqrt{d})$. The proof differs in the cases $d\not\equiv -1 \mod 4 $ and $d\equiv -1 \mod 4$. We start with the case $d\not\equiv -1 \mod 4$. Then, $\O_K=\mathbb Z+\mathbb Z\sqrt{d}$. By Lemma \ref{lemma-collapsing}, the set $S$ has to be collapsed along some vertical line $\ell_1$ and some horizontal line $\ell_2$. By the Dirichlet's theorem on prime numbers in arithmetic progressions, one can find a rational prime number $p_1$ such that $p_1=\sqrt{n}+o(\sqrt n), p_1>n+1$ and $p_1$ is prime in $\O_K$. Since $S$ is supposed to be almost uniformly equidistributed modulo $p_1$ by Lemma \ref{lemma-equidistribution}, the intersection of $S$ with any horizontal or vertical line can contain at most $p_1$ consecutive points of $\O_K$.
Therefore, $S$ has to contained in the intersection of two stripes along $\ell_1,\ell_2$ of widths $p_1$ and $p_1 \sqrt{d}$ respectively. This intersection is a rectangle containing roughly $n+o(n)$ lattice points. Since the set $S$ has $n+1$ elements, this means that it has to fill the rectangle perfectly, missing only $o(n)$ points. 

\begin{figure}
    \centering
    \includegraphics[scale=0.5, angle=270]{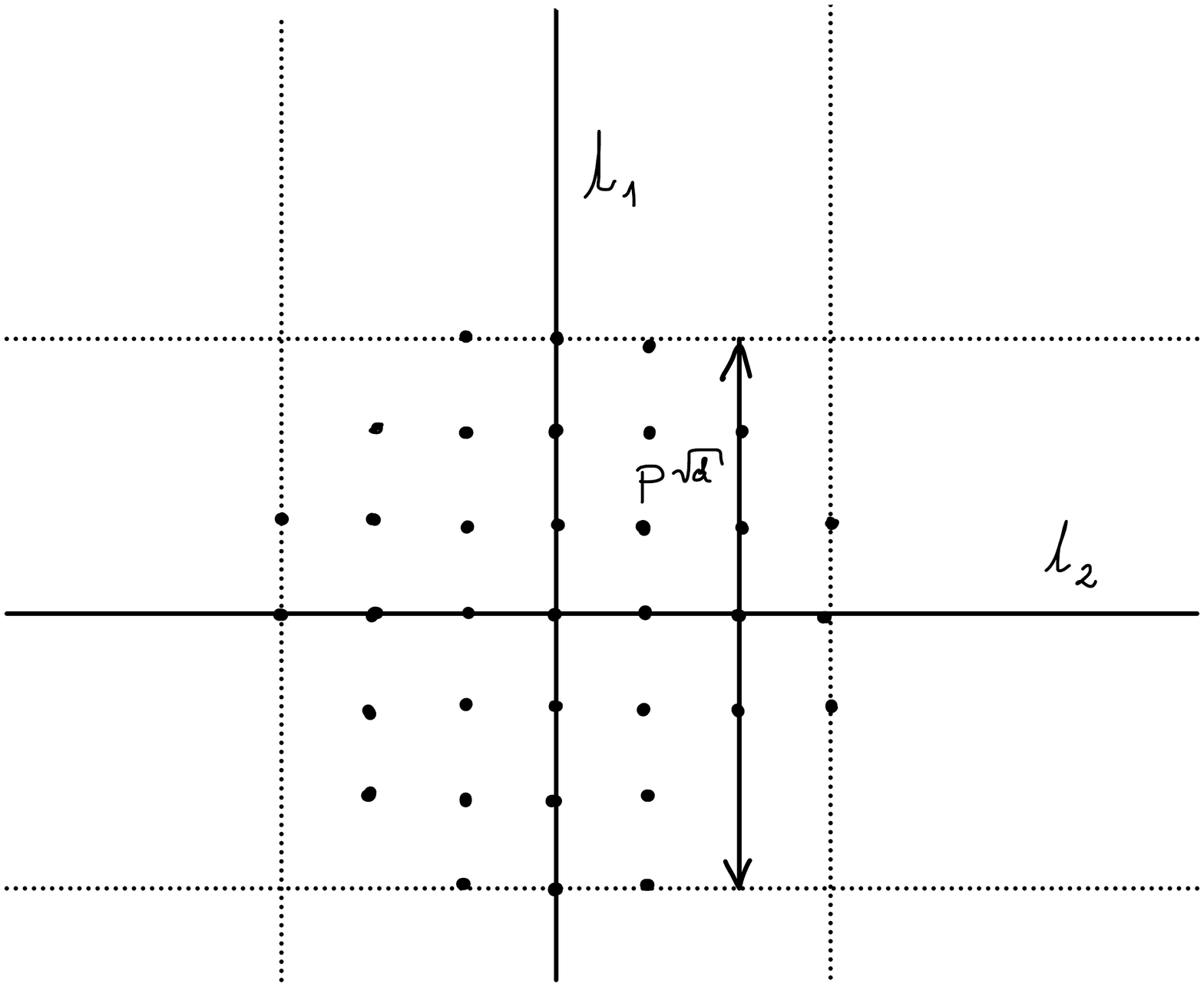} 
    \caption{$d\not\equiv -1\mod 4$}
    \label{fig:rect}
\end{figure}
\begin{figure}
    \centering
    \includegraphics[scale=0.5]{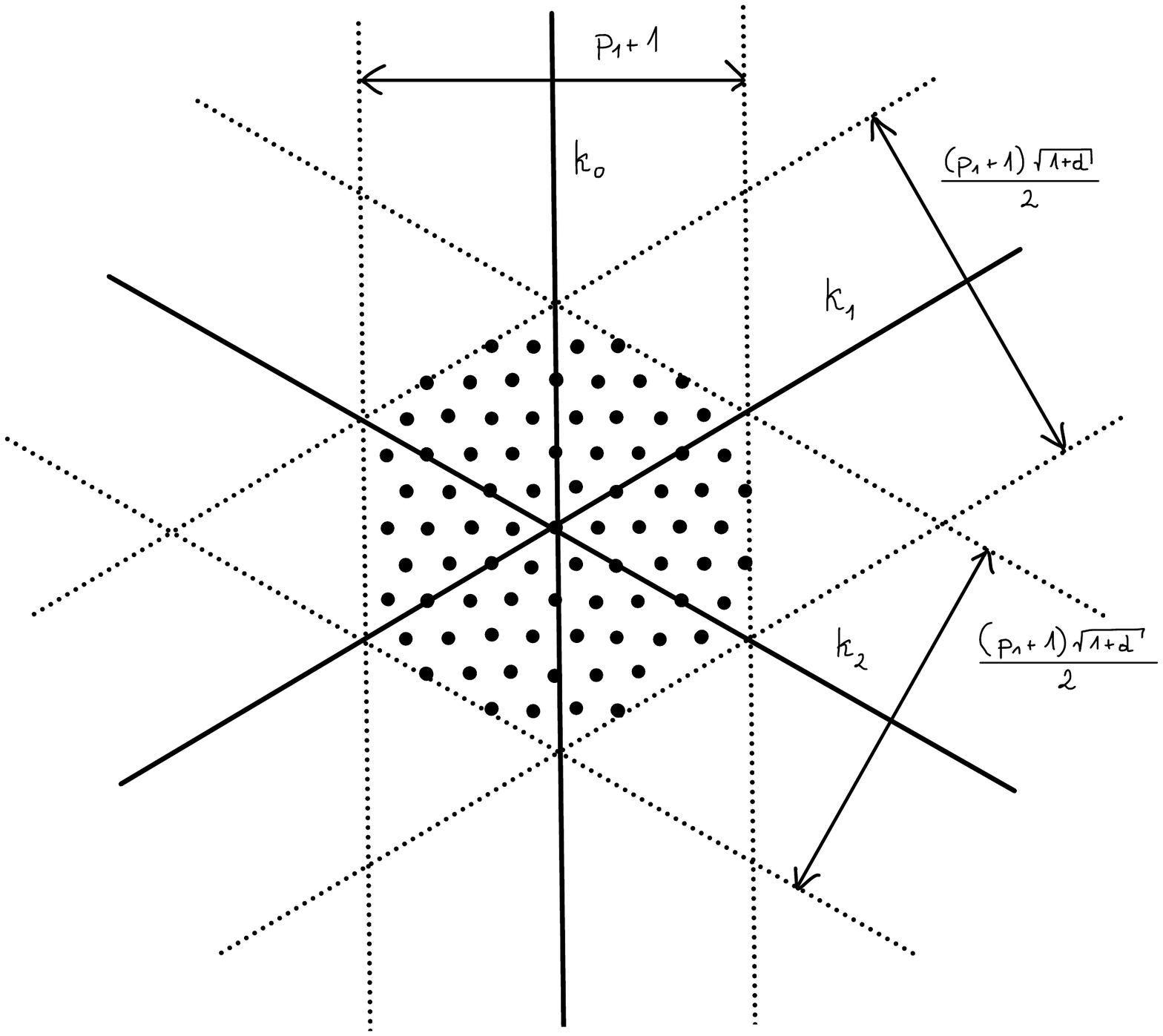}
    \caption{$d\equiv -1\mod 4$}
    \label{fig:hex}
\end{figure}

In the  case $d\equiv -1 \mod 4$, by Lemma \ref{lemma-collapsing}, $S$ is collapsed along the lines 
\begin{align*} 
k_0&\text{ parallel to }\left\{ iy | \ \ y\in\mathbb R\right\}, \\
k_1&\text{ parallel to }\left\{ x+iy| \ \ y=-\frac{x}{\sqrt{d}}\right\},\\
k_2&\text{ parallel to }\left\{ x+iy| \ \ y=\frac{x}{\sqrt{d}}\right\}. 
\end{align*} 
Again, using Lemma \ref{lemma-equidistribution}, we deduce that in the case $d\equiv -1 \mod 4$ an $n$-optimal set has to be contained in a hexagon with $n+o(n)$ points from $\O_K$. Since $S$ has $n+1$ points, this means that the set $S$ fills the hexagon perfectly, missing only $o(n)$ points. Using Dirichlet's theorem on prime numbers in arithmetic progressions, we can find a rational prime $p_1$, non-split in $\O_K$, modulo which the set $S$ fails to be almost uniformly equidistributed. The last part of the argument uses only the geometry of the rectangle and the hexagon. For an appropriate size of $p_1$, both shapes contain too many points congruent mod $p_1$ (see Figure \ref{figure-notdistr}).   
\end{proof} 
\begin{figure} 
\centering 
\includegraphics[scale=0.3]{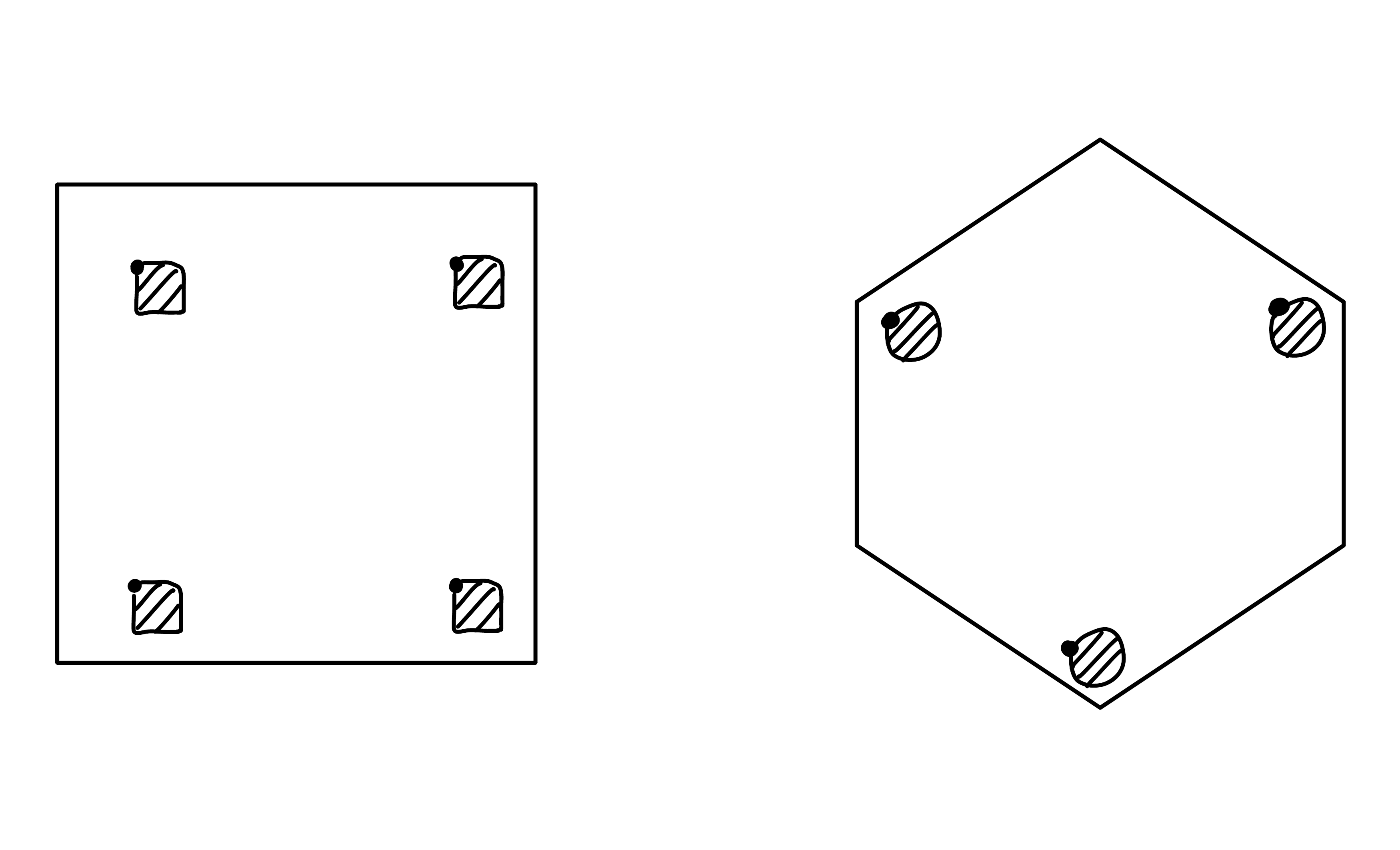} 
\caption{}
\label{figure-notdistr}
\end{figure}
\begin{remark} 
The key ingredient in the proof for the quadratic imaginary number fields is the fact that an $n$-optimal set must be collapsed in any direction. The proof of this property relies crucially on the fact that the norm $x\mapsto |N_{K/\Q}(x)|$ is a convex function. This is no longer the case in a general number field. 
\end{remark}

\section{Estimate on the energy of $n$-optimal sets. }
\label{section-energy}
Let $K$ be a number field, this time not necessarily quadratic imaginary. 
Using Proposition \ref{proposition-energy} and an estimate on the norm of the factorial ideals due to Lamoureux \cite{L}, it is possible to estimate the energy of $n$-optimal sets. The formulas will use the Euler-Kronecker constants, defined below.
\begin{definition}\cite{Iha} 
Let $K$ be a number field and let $\zeta_K(z)$ be the Dedekind zeta function of $K$. Let \begin{equation*} 
\zeta_K(z)=\frac{c_{-1}}{z-1}+c_0+c_1(z-1)+\ldots 
\end{equation*}
be the Laurent expansion of $\zeta_K$ at $s=1$. The quotient $\frac{c_0}{c_{-1}}$ is called the \textbf{Euler-Kronecker constant} of $K$ and we denote it by $\gamma _K$.  
\end{definition} 
In the case $K=\Q$, the Euler-Kronecker constant $\gamma _{\Q}$ is called the Euler-Mascheroni constant and is given by the following formula
\begin{equation*} 
\gamma_{\Q}=\lim_{n\to \infty}\left( \sum_{i=1}^n\frac{1}{i}-\log n\right).
\end{equation*}
For more information on the Euler-Kronecker constants see \cite{Iha}. 
By Proposition \ref{proposition-energy}, to estimate the energy of $n$-optimal sets it is enough to estimate $K$-factorials. Thanks to Lamoureux \cite{L} we have the following estimate.
\begin{theorem}\cite{L} 
Let $K$ be a number field. Then,
\begin{equation*} 
\log N_{K/\Q}(n!_K)=n\log n -n(1+\gamma_K-\gamma_{\Q})+o(n).
\end{equation*}
\end{theorem} 
Using Proposition \ref{proposition-energy} we deduce
\begin{corollary}\cite{BFS2017}
\label{corollary-energyforsets}Let $K$ be a number field and let $\O_K$ be its ring of integers. Let $S$ be an $n$-optimal subset of $\O_K$. Then
\begin{equation*} 
\log N_{K/\Q}(E(S))=n^2\log n -n^2(\frac{3}{2}+\gamma _K-\gamma_{\Q})+o(n^2).
\end{equation*}
Moreover, for every subset $T\subseteq \O_K$ with $|T|=n+1$ we have
\begin{equation*} 
\log N_{K/\Q}(E(T))\geq n^2\log n-n^2(\frac{3}{2}+\gamma_K-\gamma_{\Q})+o(n^2). 
\end{equation*} 
\end{corollary}

\section{$n$-optimal sets for an arbitrary number field}
\label{section-numberfields}
As we already mentioned, the methods for quadratic imaginary number fields cannot be adapted to the case of a general number field, as in general the norm is not convex. In this section we sketch the proof for an arbitary number field \cite{FS2018}. 
Fix a number field $K$ of degree $N$ and write $\O_K$ for the ring of integers.

We argue by contradiction. Assume there exists a sequence $S_{n_i}$ of $n_i$-optimal sets where $n_i$ tends to infinity. Let $V=K\otimes_{\Q}\R\cong \R^{r_1}\times \C^{r_2}$. Write $d=r_1+r_2$. The absolute value of the norm $|N_{K/\Q}(\cdot)|$ extends to the map $\|\cdot\|\colon V\to\mathbb R$ defined by $\|v\|=\prod_{i=1}^{r_1}|v_i|\prod_{i=r_1+1}^{d}|v_i|^2$ for $v=(v_1, \ldots ,v_d)$. 

\subsection{Enclosure of $n$-optimal sets in cylinders}
The first step is to show that we can enclose $S_{n_i}$ in a cylinder of the volume $n_i+o(n_i)$. 
\begin{definition} 
A cylinder $\mathcal{C}$ in $V$ is a coordinate-wise product of balls:
\begin{equation*} 
\mathcal{C}=\prod_{i=1}^{r_1}B_{\R}(x_i,r_i)\times \prod_{i=r_1+1}^{d}B_{\C}(x_i,r_i),
\end{equation*} 
where $x_i\in \R$ for $i=1, \ldots , r_1$, $x_i\in \C$ for $i=r_1+1,\ldots ,d$ and $r_i\in \R_{\geq 0}$ for $i=1, \ldots ,d.$
\end{definition}
The volume of a cylinder $\mathcal{C}$ is defined as its Lebesgue measure. 

\begin{theorem}{\cite[Theorem 3.1]{FS2018}}
\label{theorem-shape}
There exists a positive constant $\theta$ dependent only on $K$ such that for every $n$-optimal set $S\subseteq \O_K$ there exists a cylinder $\mathcal{C}$ of volume $\theta n$ with $S\subseteq \mathcal{C}$. 
\end{theorem}
And as a consequence we get:
\begin{corollary} 
\label{corollary-cylinder}
There exists a positive constant $A>0$ depending only on $K$ such that the set $\Omega =B_{\R}(0,A)^{r_1}\times B_{\C}(0,A)^{r_2}$ has the following property. Let $S\subseteq \O_K$ be an $n$-optimal set. Then, there exist $s,t\in V$ such that $||s||=n|\Delta_K|^{1/2}$ and $s^{-1}(S-t)\subseteq \Omega$.   
\end{corollary}
Theorem \ref{theorem-shape} was implicit in the proof of Theorem \ref{theorem-main} for $K=\Q(i)$ \cite{PV} and for $K$ quadratic imaginary \cite{BFS2017}. Indeed, in Section \ref{section-quadratic imaginary} we remark that $S$ is contained in a convex polygon of volume $|\Delta_K|^{1/2}n+o(n)$ which can be always enclosed in a cylinder of volume $\theta n$. It was proved using the collapsing procedure which heavily relies on the fact that the norm $N_{K/\Q}$ is convex for imaginary quadratic number field $K$. 
Proving Theorem \ref{theorem-shape} for a general case was one of the main difficulties in proving Theorem \ref{theorem-main}. One of the main ingredients is a result on counting the number of $x\in \O_K$ such that $\|x(a-x)\|\leq X^2$ for some $X>0$ and $a\in \O_K$ such that $\|a\|\geq Xe^{-B}$ where $B\in\R$ is fixed. 
This result, in a sense, substitutes for the role of the convexity of the norm. To give a precise statement of the result we introduce the notion of a good fundamental domain. 
\begin{definition} 
A \textbf{good fundamental domain} of $\O_K^{\times}$ in $V^{\times}$ is a set $\mathcal{F}$ which is a finite union of convex closed cones in $V^{\times}$ such that $\mathcal{F}/\R^{\times}$ is compact in the projective space $\mathbb{P}(V)$, $V^{\times}=\bigcup_{\lambda\in\O_K^{\times}}\lambda\mathcal{F}$, $\mathrm{int}\mathcal{F}\cap \lambda(\mathrm{int}\mathcal{F})=\emptyset$ for every $\lambda \in\O_K^{\times}$, $\lambda\neq 1$ and $\partial\mathcal{F}$ does not contain non-zero points from $\O_K$. 
\end{definition}

Fix a good fundamental domain $\mathcal{F}$. Let $a\in \O_K$, $a\neq 0$ and $X>0$. Define
\begin{equation*} 
S(a,X)=\{(x,\lambda)\in (\mathcal{F}\cap\O_K)\times \O_K^{\times}| \ \ \|x(a-x\lambda^{-1})\|\leq X^2,\ \ \|x\|\leq X\}. 
\end{equation*}
Denote $\logp x:=\log x$ if $x>1$ and $\logp x:=0$ otherwise.  
\begin{proposition}{\cite[Propositon 2.5]{FS2018}} 
Let $K$ be a number field of degree $N$ with $d$ Archimedean places. Let $B\in \R$. Let $\kappa=\frac{1}{3}$ if $N=1$ and $\kappa=\min\left\{\frac{1}{2N(N-1)},\frac{1}{4N-1}\right\}$ otherwise. Fix a good fundamental domain $\mathcal{F}$. There exist constants $\theta_1, \theta _2, \theta _3, \theta_4$ dependent only on $K,B$ and $\mathcal{F}$ such that for every $X>0$ and $a\in\O_K$ such that $\|a\|\geq Xe^{-B}$ we have
\begin{enumerate} 
\item $|S(a,X)|\leq \theta_1X^{1+\kappa}\|a\|^{-\kappa}+\theta_2(\log X)^{2d-2}+\theta_3\logp\logp\logp\log \|a\|+\theta_4$.
\item Suppose $a\in \mathcal{F}$. Then, for every $\varepsilon>0$ there exists $M$ such that 
\begin{align*} 
&|\{(x,\lambda)\in S(a,X)| \ \ \|\lambda\|_{\infty}\geq M\}|\leq\\ &\varepsilon X^{1+\kappa}\|a\|^{-\kappa}+\theta_2(\log X)^{2d-2}+\theta_3\logp\logp\logp\log\|a\|+\theta_4.
\end{align*} 
\end{enumerate}
\end{proposition} 
The proof of the proposition is based on Aramaki-Ikehara Tauberian theorem \cite{Aramaki}, Baker-W\"ustholz's inequality on linear forms in logarithms \cite[Theorem 7.1]{BW} and on some elementary estimates on the number of integer points in cylinders. 

As a consequence we obtain the following result which may be of independent interest. 
\begin{theorem} 
Let $K$ be a number field of degree $N$ with $d$ Archimedean places. Let $B\in \R$ and let $\kappa=\frac{1}{3}$ if $N=1$ and $\kappa =\min\left\{\frac{1}{2N(N-1)},\frac{1}{4N-1}\right\}$ otherwise. There exist constants $C_1,C_2,C_3,C_4$ dependent only on $K$ and $B$ such that for every $X>0$ and $a\in\O_K$ such that $\|a\|\geq Xe^{-B}$ we have
\begin{align*} 
&|\{x\in\O_K| \ \ \|x(a-x)\|\leq X^2\}|\leq \\ &C_1X^{1+\kappa}\|a\|^{-\kappa}+C_2(\log X)^{2d-2}+C_3\logp\logp\logp\log\|a\|+C_4.
\end{align*}
\end{theorem}

\subsection{Limit measures}
\label{subsection-limitmeasures}
For the sake of contradiction, we assumed that there exists a sequence of $n_i$-optimal sets $(S_{n_i})$ in $\O_K$ where $n_i$ tends to infinity. Theorem \ref{theorem-shape} shows that, up to translation and suitable rescaling, all of the sets $S_{n_i}$ can be enclosed in one compact set. Using that fact we construct a tight family of measures associated to sets $S_{n_i}$ and consider their weak-$*$ limits. In this section we study properties of such limits. 

By Corollary \ref{corollary-cylinder}, there exist sequences $(s_{n_i})$, $(t_{n_i})\subseteq V$ with $\|s_{n_i}\|=n_i|\Delta_K|^{1/2}$ and a compact set $\Omega$ such that $s_{n_i}^{-1}(S_{n_i}-t_{n_i})\subseteq \Omega$. Define the measures
\begin{equation*} 
\mu _{n_i}:=\frac{1}{n_i}\sum _{x\in S_{n_i}}\delta_{s_{n_i}^{-1}(x-t_{n_i})}. 
\end{equation*} 
Since $\Omega$ is compact we can consider, passing to a subsequence if necessary, a weak-* limit of $\mu _{n_i}$. Existence of such limits crucially uses the fact that $\Omega$ is compact. In this section we study the properties of such weak-* limits. 
\begin{definition} 
A probability measure $\mu$ on $V$ is a \textbf{limit measure} if it is a weak-* limit measure of the measures $\mu_{n_i}$ defined above. 
\end{definition} 
A limit measure is a probability measure supported on $\Omega$, absolutely continuous with respect to the Lebesgue measure and of density at most one (see\cite[Lemma 5.2]{FS2018}). 
The measure $\mu$ encodes some information about the large scale geometry of the sets $S_{n_i}$. The idea for the rest of the proof is to use the properties of $n$-optimal sets to show that such limit measures cannot exist.    

By analogy with the finite subsets of $\O_K$, one can define the energy of compactly supported, probability measures on $V$, absolutely continuos with respect to the Lebesgue measure and of bounded density. 
\begin{definition} 
Let $\nu$ be a compactly supported measure on $V$, absolutely continuous with respect the Lebesgue measure and of bounded density. The energy of $\nu $ is given by
\begin{equation*} 
I(\nu)=\int_{V}\int_{V}\log\|x-y\|d\nu(x)d\nu(y).
\end{equation*} 
\end{definition} 
The estimate on the energy of $n$-optimal sets (Corollary \ref{corollary-energyforsets}) leads to a formula for the energy of limit measures.
\begin{proposition}{\cite[Proposition 5.3]{FS2018}}
\label{proposition-energylimitmeasure}
Let $\mu$ be a limit measure. Then,
\begin{equation*} 
I(\mu)=-\frac{1}{2}\log |\Delta _K|-\frac{3}{2}-\gamma_K+\gamma_{\Q}. 
\end{equation*}
\end{proposition}

Using the estimate on the energy of $n$-optimal sets (Corollary \ref{corollary-energyforsets}) and the fact that $n$-optimal sets are energy minimizing, one can give a lower bound on the energy of a compactly supported probability measure on $V$ with density at most $1$. 
\begin{lemma}{\cite[Lemma 5.4]{FS2018}} 
\label{lemma-energylowerbound}
Let $\nu$ be a compactly supported probability measure on $V$ with density at most $1$. Then
\begin{equation*} 
I(\nu)\geq -\frac{1}{2}\log |\Delta_K|-\frac{3}{2}-\gamma_K+\gamma_{\Q}.
\end{equation*} 

\end{lemma}
The idea of the proof of this estimate is to construct a sequence of sets $E_n$, $|E_n|=n+1$ such that the rescaled normalized counting measures $\frac{1}{n}\sum_{x\in E_n}\delta_{n^{-1/N}|\Delta_K|^{-1/2N}x}$ weakly-* converges to $\nu$. Then, one can relate the asymptotic growth of the energies of $E_n$ with the energy of $\nu$ and use Proposition \ref{proposition-energy}. 

Lemma \ref{lemma-energylowerbound} together with Proposition \ref{proposition-energylimitmeasure} implies that limit measures minimize the energy among all compactly supported probability measures on $V$ of density at most 1. This gives a strong constrains on the structure of limit measures.
\begin{proposition}{\cite[Proposition 5.5]{FS2018}}
\label{proposition-5.5}
Let $\nu$ be a compactly supported probability measure on $V$ of density at most 1 which is realizing the minimal energy among all such measures. Then, there exists an open set $U$ and $v\in V$ such that 
\begin{enumerate} 
\item $\nu=\Leb|_U$,
\item $\lambda(\overline{U}-v)\subseteq U-v$ for every $0\leq \lambda <1$,
\item $(\partial U-v)\cap V^{\times}$ is a codimension 1 submanifold of $V^{\times}$ of class $C^1$.
\end{enumerate} 
\end{proposition}
The proof uses a procedure of collapsing measures which is a continuous version of the discrete collapsing. As opposed to the quadratic imaginary case, here one can collapse the measure only along the hyperplanes parallel to a hyperplane contained in $V\setminus V^\times.$

Since any limit measure $\mu$ minimizes the energy among all compactly supported probability measures on $V$ of density at most $1$, $\mu$ satisfies the conclusion of Proposition \ref{proposition-5.5}. In particular, the set $S_{n_i}$ is equal to $(s_{n_i}U+t_{n_i})\cap \O_K$, modulo $o(n)$ points. In the remainder of the argument, the sets $(s_{n_i}U+t_{n_i})$ will play the role of the rectangles or the hexagons from the proof in the quadratic imaginary case. 

\subsection{Discrepancy} 
In the last step of the proof one needs to find a prime power $\frak p^l\subset \O_K$ modulo which the set $S_{n_i}$ fails to be almost uniformly equidistributed. In the quadratic imaginary case, it was possible to describe the shape of $S_{n_i}$ quite explicitly and find the prime power by hand. This is no longer possible, since one only knows that the shape is given by the set $U$ from Proposition \ref{proposition-5.5}. The existence of a prime power with the desired properties can be shown using the discrepancy of the set $U$.
\begin{definition} 
Let $W$ be a bounded measurable subset of $V$. For $x\in V^\times$, $v\in V$ define $N_x(W,v):=|(xU+v)\cap\O_K|$. We define the discrepancy as
\begin{equation*} 
D_x(W,v):=N_x(W,v)-|\Delta _K|^{-\frac{1}{2}}\Leb(W)\|x\|
\end{equation*} 
and the maximal discrepancy 
\begin{equation*} 
D_x(W):=\esssup _{v\in V}|D_x(W,v)|. 
\end{equation*}
\end{definition}
Using the fact that $n$-optimal sets are almost uniformly equidistributed modulo every power of every prime ideal in $\O_K$ together with a version of the prime number theorem one can give a uniform upper bound on $D_x(U)$ for all $x\in V^\times$. 
\begin{lemma}{\cite[Lemma 6.3]{FS2018}}
\label{lemma-upperboundondiscrepancy} 
Let $\mu$ be a limit measure on $V$. Let $U$ be a non-empty open bounded subset of $V$ such that $\mu=\Leb|_U$ and $\partial U$ is Jordan measurable of Jordan measure $0$. Then $D_x(U)<1$ for all $x\in V^{\times}$. 
\end{lemma} 
On the other hand, using the smoothness of the boundary of $U$, we have:
\begin{lemma}\cite[Lemma 6.4]{FS2018} 
\label{lemma-lowerboundondiscrepancy}
Assume $V=\R^{r_1}\times \C^{r_2}$ with $r_1+2r_2>1$. Let $W$ be an open bounded subset of $V$ such that $\partial W\cap V^{\times}$ is a submanifold of $V^\times $ of class $C^1$ and $\lambda \overline{U}\subseteq U$ for every $0\leq \lambda <1$. Then, there exists $x\in V^\times $ such that $D_x(U)>1$. 
\end{lemma}

\subsection{Proof of Theorem \ref{theorem-main}}
In this section we gather results described earlier to sketch the proof that large $n$-optimal sets do not exist. 
\begin{proof}[Proof of Theorem \ref{theorem-main}]
Let $V=K\otimes _{\Q} \R=\R^{r_1}\times \C^{r_2}$. For the sake of contradiction let us assume that there exists a sequence of $n_i$-optimal sets $S_{n_i}\subseteq \O_K$ with $n_i\to \infty$. By Corollary \ref{corollary-cylinder}, there exists a compact set $\Omega\subseteq V$ and sequences $(s_{n_i})_{i\in\N},(t_{n_i})_{i\in \N}\subseteq V$ with $\|s_{n_i}\|=n_i|\Delta_K|^{1/2}$ such that $s_{n_i}^{-1}(S_{n_i}-t_{n_i})\subseteq \Omega$. We define 
\begin{equation*} 
\mu _{n_i}=\frac{1}{n_i}\sum_{x\in S_{n_i}}\delta_{s_{n_i}^{-1}(x-t_{n_i})}.
\end{equation*}
Since $\Omega$ is compact, after passing to a subsequence if necessary, these measures converge weak-* to a probability measure $\mu$. This is a limit measure. By Section \ref{subsection-limitmeasures}, $\mu$ is compactly supported, absolutely continuous with respect to the Lebesgue measure of density at most $1$, so by Propositon \ref{proposition-energylimitmeasure} and Lemma \ref{lemma-energylowerbound} it minimizes the energy among all such measures. By Proposition \ref{proposition-5.5}, $\mu=\Leb|_U$ where $U$ is an open set with piecewise $C^1$ boundary. Finally using Lemma \ref{lemma-upperboundondiscrepancy} and Lemma \ref{lemma-lowerboundondiscrepancy} we get a contradiction, since the discrepancy would be at the same time bigger and smaller than $1$.
\end{proof}

\section{Open problems}
\label{section-openproblems}


\subsection{Function fields} Let $\mathbb F_q$ be a finite field with $q$ elements, $\mathbb F_q=\{a_0,\ldots, a_{q-1}\}$. Let $n\in\N$. Write $n=\sum_{i=0}^k d_i q^{i}$ with $d_i\in\{0,\ldots, q-1\}$. It was observed by Bhargava in \cite{Bh1} that the sequence $s_n:=\sum_{i=0}^k a_{d_i} t^i\in \mathbb F_q[t]$ is a simultaneous $p$-ordering in $\mathbb F_q[t]$.

More generally one can consider a protective curve $\mathbf{C}$ over $\mathbb F_q$, select a finite set $\Sigma \subset \mathbf{C}(\mathbb F_q)$ and consider the ring $\Gamma(\mathbf{C}\setminus \Sigma, \O)$ of the regular functions on $\mathbf{C}\setminus \Sigma$. The case of the ring $\mathbb F_q[t]$ is recovered by taking $\mathbf C=\mathbb P^1$ and $\Sigma$ consisting of the point at infinity. 
\begin{question} 
For which curves and sets $\Sigma$ does the ring $\Gamma(\mathbf{C}\setminus \Sigma, \O)$ admits a simultaneous $\mathfrak{p}$-orderings or arbitrarily large $n$-optimal sets?
\end{question} 
It seems that the answer might depend of the genus of the curve, with high genus unlikely to contain large $n$-optimal sets.

\subsection{Schinzel's problem}
The question on existence of simultaneous $\mathfrak{p}$-orderings seems a bit similar to an old problem called Schinzel's problem:
\begin{question}[Schinzel's problem \cite{Sch1969}]
\label{question-Schinzel}
Let $K\neq \Q$ be a number field and let $\O_K$ be its field of fractions. Does there exist a sequence $(a_i)_{i\in \N}\subseteq \O_K$ such that for every ideal $I$ in $\O_K$ with the norm $N(I)$, the sequence $a_0,\ldots, a_{N(I)-1}$ is a complete system of representatives of $\O_K/I$.  
\end{question} 
There have been partial results obtained: Wantula (1969) (unpublished) showed that $K$ cannot be a quadratic number field and Was\'en (1976) \cite{Wasen} showed that $\O_K$ has to be a principal ideal domain. 


\begin{definition} 
A sequence $(a_i)_{i\in\N}$ satisfying the condition from Question \ref{question-Schinzel} is called a Schinzel sequence. 
\end{definition}
Frisch \cite{FCC} showed that if $\O_K$ admits a Schinzel sequence, then it is Euclidean. 
\begin{proposition}[\cite{FCC}, p.102]
Let $K$ be a number field and let $\O_K$ be its ring of integers. Let $N(x)$ denote the norm of the principal ideal $(x)$ in $\O_K$. If $\O_K$ admits a Schinzel sequence, then $\O_K$ is Euclidean with the norm $N$.
\end{proposition}
\begin{proof} 
Assume there exists a Schinzel sequence $(a_i)_{i\in \N}$ in $\O_K$. We can assume $a_0=0$. By the definition of Schinzel sequence, $N(a_j)\leq j$ for every $j\in\N\setminus\{0\}$. Let $a,b \in \O_K$ with $a\neq 0$ and $a$ not a unit. By the definition of a Schinzel seqqunece, there exists $m\in \{0,\ldots , N(a)-1\}$ such that $b=ca+a_m$. As observed before, $N(a_m)\leq m<N(a)$.    
\end{proof} 
Could the methods used to prove Theorem \ref{theorem-main} be used to address the Schinzel's problem? If one could construct limit measures attached to a Schinzel sequence, what would be their properties?


\begin{thebibliography}{9}
\bibitem{AC} David Adam and Paul-Jean Cahen, \textit{Newtonian and Schinzel quadratic fields}, Journal of Pure and Applied Algebra \textbf{215} (2011), no. 8, 1902--1918.

\bibitem{Aramaki} Junichi Aramaki, \textit{On an extension of the Ikehara Tauberian theorem}, Pacific Journal of Mathematics \textbf{133} (1988), no. 1, 13--30.

\bibitem{BW} Alan Baker and Gisbert W\"ustholz, \textit{Logarithmic Forms and Diophantine Geometry}, Cambridge University Press, 2008. 

\bibitem{Bh1}Manjul Bhargava, \textit{P-orderings and polynomial functions on arbitrary subsets of Dedekind rings}, Journal f\"ur die reine und angewandte Mathematik \textbf{490} (1997), 101--128.

\bibitem{Bh2} Manjul Bhargava, \textit{The factorial function and generalizations}, The American Mathematical Monthly \textbf{107} (2000), no. 9, 783--799. 

\bibitem{BFS2017}Jakub Byszewski, Mikolaj Fraczyk and Anna Szumowicz, \textit{Simultaneous $\mathfrak{p}$-orderings and minimizing volumes in number fields}, Journal of Number Theory \textbf{173} (2017), 478--511.

\bibitem{FCC} Paul-Jean Cahen and Jean-Luc Chabert, \textit{Old problems and new questions around integer-valued polynomials and factorial sequences} In James W. Brewer, Sarah Glaz, William Heinzer and Bruce Olberding, Multiplicative ideal theory in commutative algebra: A Tribute to the Work of Robert Gilmer. Springer (2006)
\bibitem{CC} Paul-Jean Cahen and Jean-Luc Chabert, \textit{Test sets for polynomials: n-universal subsets and Newton sequences}, Journal of Algebra \textbf{502} (2018), 277--314.

\bibitem{Iha} Yasutaka Ihara, \textit{On the Euler-Kronecker constants of global fields and primes with small norms}, Algebraic geometry and number theory, 2006, pp. 407--451.

\bibitem{Sch1969} W. Narkiewicz, \textit{Some unsolved problems}, Bull. Soc. Math France, M\'{e}moire \textbf{25} (1971), 159--164.

\bibitem{L} Matthew Lamoureux, \textit{Stirling's Formula in Number Fields}, Doctoral Dissertations, University of Connecticut (2014).

\bibitem{FS2018}Anna Szumowicz and Mikolaj Fraczyk, \textit{On the optimal rate of equidistribution in number fields}, available at: \url{https://arxiv.org/pdf/1810.11110.pdf}

\bibitem{PV} Vladislav Volkov and Fedor Petrov, \textit{On the interpolation  of integer-valued polynomials}, Journal of Number Theory \textbf{133} (2013), no. 12, 4224--4232.

\bibitem{Wasen} Rolf Was\'en,
\textit{On sequences of algebraic integers in pure extensions of prime degree.}, Colloquium Mathematicum. Vol. 30.  No. 1. Institute of Mathematics Polish Academy of Sciences (1974), pp. 89--104 

\bibitem{Wood}Melanie Wood, \textit{P-orderings: a metric viewpoint and the non-existence of simultaneous orderings}, Journal of Number Theory \textbf{99} (2003), no. 1, 36--56. 
\end{thebibliography}
\end{document}